\documentclass[a4paper,10pt]{article}

\usepackage{indentfirst,amsmath,amsfonts,amstext,amssymb,amscd,bezier,amsthm}
\usepackage[latin1]{inputenc}
\usepackage[dvips]{graphicx}
\usepackage[latin1]{inputenc}
\usepackage{setspace}
\usepackage{cases}
\usepackage{graphicx}
\usepackage{tabularx}
\usepackage{hyperref}
\usepackage{longtable}
\usepackage[usenames,dvipsnames]{pstricks}
\usepackage{epsfig}
\usepackage{pst-grad}
\usepackage{pst-plot}
\usepackage[all]{xy}
\usepackage{hyperref}
\usepackage[T1]{fontenc}
\usepackage{ae,aecompl}
\usepackage{indentfirst}
\usepackage{xspace}
\usepackage{color}
\usepackage{psfrag}
\usepackage{mathrsfs}
\usepackage{upgreek}                            % letra grega nao fica em italico
\usepackage{makeidx}                            % �ndice remissivo \makeindex
\usepackage{fancyhdr}
\usepackage{fancybox}
\usepackage[rm]{titlesec}
\usepackage{float}   
\usepackage{enumerate}     

\input xy
\xyoption{all}
\newcommand{\F}{\mathbb{F}}

\newcommand{\ff}{\mathcal{F}}
\newcommand{\cF}{\mathcal{F}}

\newcommand{\cH}{\mathcal{H}}

\newcommand{\Aut}{\mathrm{Aut}}

\newcommand{\kk}{\mathbb K}

\newcommand{\Z}{\mathbb Z}

\newcommand{\divi}{\operatorname{div}}
\newcommand{\aut}{\operatorname{Aut}}
\newcommand{\supp}{\operatorname{Supp}}

\theoremstyle{plain}
\newtheorem{thm}{Theorem}[section]
\newtheorem{defi}[thm]{Definition}
\newtheorem{prop}[thm]{Proposition}
\newtheorem{lem}[thm]{Lemma}
\newtheorem{cor}[thm]{Corollary}
\newtheorem{rem}[thm]{Remark}

\DeclareMathOperator{\sym}{Sym}
\DeclareMathOperator{\im}{Im}

 \font\numberfont= pzcmi scaled
3000
\titleformat{\chapter}[display]
  {\normalfont\Large % roman
  }
  {%\titlerule[3pt]%
   \filright
   \rule[32pt]{.7\linewidth}{4pt}
   \hspace{-8pt}
   \shadowbox{
   \begin{minipage}{.15\linewidth}
     \begin{center}
          \textsl{\bf {\large \chaptertitlename}}\\
       \vspace{1ex}
       {\bf {\numberfont \thechapter}}\\
       \vspace{1ex}
     \end{center}
   \end{minipage}}
  }
  {-10pt}
  {\filcenter
           \sl
           \bf
              \Huge
     }
  [\vspace{-1cm}\hfill\rule{.8\textwidth}{0.5pt}\\
\vskip-2.8ex\hfill\rule{.7\textwidth}{4pt}\vspace*{-1ex}]
\titlespacing{\chapter}{0pt}{*4}{*1}

\titleformat{\section}[block]
{\normalfont\bfseries} {\thesection}{0.5em}{}

\titleformat{\subsection}[block]
{\normalfont\large\bfseries} {\thesubsection}{0.5em}{}

\makeindex                                      % comando para fazer o �ndice remissivo
\makeatletter
\numberwithin{equation}{section}

\begin{document}

%%%%%%%%%%%%%%%%%%%%%title%%%%%%%%%%%%%%%%%
\title{The Hurwitz curve over a finite field and its Weierstrass points for the  morphism  of lines}

\author{\textbf{Nazar Arakelian} \\
 \small{CMCC, Universidade Federal do ABC, Santo Andr\'e, Brazil}  \\
 \textbf{Herivelto Borges}\\
 \small{ICMC, Universidade de S\~ao Paulo, S\~ao Carlos, Brazil}\\
 \textbf{Pietro Speziali}\\
 \small{ICMC, Universidade de S\~ao Paulo, S\~ao Carlos, Brazil}
}

\maketitle
%%%%%%%%%%%%%%%%%%%%%%%%%%%%%%%%%%%%%%%%%%%%%%%%%%%%%%%%%%%%%%%%%%%%%%%%%%%%%%%%%%%%%%%%%%%%%%%%%%%%%%%%%%%%%%%%%%%%%%%%%%%%%%%%%%%%%%%%
 \begin{abstract}
For  any  smooth Hurwitz curve $\mathcal{H}_n: \, XY^n+YZ^n+X^nZ=0$ over the finite field $\F_{p}$,  an  explict  description of  its Weierstrass points for the  morphism  of lines is  presented. As a consequence, the  full automorphism  group $\Aut(\mathcal{H}_n)$, as well as the genera of all Galois subcovers  of $\mathcal{H}_n$,  with  $n\neq 3, p^r$, are computed. Finally, a question by F. Torres on plane  nonsingular maximal curves is answered. 
 \end{abstract}

%%%%%%%%%%%%%%%%%%%%%%%%%%%%%%%%%%%%%%%%%%%%%%%%%%%%%%%%%%%%%%%%%%%%%%%%%%%%%%%%%%%%%%%%%%%%%%%%%%%%%%%%%%%%%%%%%%%%%%%%%%%%%%%%%%%%%%%%%

\section {Introduction}\label{intro}

In many branches of mathematics, the Klein quartic $\mathcal{H}_3:\,  XY^3+YZ^3+X^3Z=0$
is a famous example of a curve with remarkable geometric and arithmetic properties (\cite{El},  \cite{Klein},  and  \cite{Shi}). 
It is known,  for instance,   that  over the field of  complex numbers  $\mathcal{H}_3$ has  $168$ automorphisms,
  and it   is the unique curve of genus $3$ attaining the
 Hurwitz bound     $|\Aut (\mathcal{X})|\leq 84(g_{\mathcal{X}}-1)$,  $g_{\mathcal{X}}\geq 2$ \cite{Wiman}. In positive characteristic, however,  the Hurwitz bound  may not be valid due to the possibility of wild ramification. An example is the fact that   $\mathcal{H}_3$ over  $\overline{\mathbb{F}}_3$   has $6048$ automorphisms. 
  Exceptional results  that  may  occur only in  positive characteristic  make the theory of curves over finite fields into a  source  of compelling problems.
 Some of  these results  impact not only in the theory itself, but also   related areas such as finite geometry, coding theory  and  number theory.

A natural generalization  of the Klein quartic is the so-called  Hurwitz curve  
$$\mathcal{H}_n:\,  XY^n+YZ^n+X^nZ=0,$$
 where  $n\geq 3$.  Over the finite field $\mathbb{F}_q$, where $q$ is a power of a prime $p\nmid n^2-n+1$, the   curve  $\mathcal{H}_n$    is smooth, and it
 has been  investigated from many points of view  (\cite{AKT}, \cite{CH} and \cite{GD}). For instance, it  is well known that  the curves $\mathcal{H}_n$ for which $n^2-n+1$ divides $q+1$ are $\mathbb{F}_{q^2}$-maximal,  i.e.,  they meet the Hasse-Weil upper bound  (\cite{AKT}, \cite{CH}).
 
 For  the smooth curve  $\mathcal{H}_n$ over a finite field,  the primary goal of this study  is to characterize  its   Weierstrass points for the  morphism  of lines. That is,  the paper  will focus on  the  special  set  of  inflection points  $P\in  \mathcal{H}_n$  for which  the intersection multiplicity  $I(P,\mathcal{H}_n \cap T_{P}\mathcal{H}_n)$ is somewhat  large.  In general, the complete characterization of this special set is  highly  desired  as it has  direct  applications in a range of topics, such as  finite geometry,  coding theory,  St\"ohr-Voloch theory and   Galois points theory.  In this manuscript, an interesting application  will be the computation  of  the full automorphism  group $\Aut(\mathcal{H}_n)$, as well as the genera of all Galois subcovers  of $\mathcal{H}_n$,  with  $n\neq 3, p^r$.
 
 A further application of our results  is related  to maximal curves. In some detail, we answer a question raised by Fernando Torres during the "Workshop on Algebraic curves and Function Fields over a Finite Field" held in Perugia in February 2015. The question was whether  plane nonsingular maximal curves that are not isomorphic neither to a Fermat nor  Hurwitz curve do exist or not. We give a positive answer to this question in Section 6, where a family of such curves, constructed via Lucas-type polynomials, is presented.

\subsection{Notation}\label{not}
Here, we fix some notation. Henceforth throughout the  text, 
\begin{itemize}
\item $p$ is  a prime  number, $\F_{p}$ is the corresponding  finite field, and $\mathbb{K}$ is the algebraic closure of $\F_{p}$
\item   the  integer $n>2$   is  such  that  $p\nmid n^2-n+1$
\item  $\mathcal{H}_n: XY^n+YZ^n+X^nZ=0$ is the  smooth Hurwitz curve defined over $\F_{p}$
\item   the function  field of  $\mathcal{H}_n$ is denoted by  $\mathbb{K}(x,y)$,   where $xy^n+y+x^n=0$
\item    $\Aut_{\mathbb{K}}(\mathcal{H}_n)$ denotes the full  automorphism group of  $\mathcal{H}_n$
\item for each point  $P\in  \mathcal{H}_n$,  $j(P):=I(P, \mathcal{H}_n  \cap T_{P}\mathcal{H}_n)$  denotes the intersection multiplicity of $\mathcal{H}_n$ and the  tangent line  $T_{P}\mathcal{H}_n$ at $P$
\item  $\Omega=\{P_1,P_2,P_3\} \subseteq \mathcal{H}_n$, where $P_1=(1:0:0)$, $P_2=(0:1:0)$ and $P_3=(0:0:1)$.

\end{itemize}

\section{Preliminaries}\label{back}
Let $\ff :  \, F(X,Y,Z)=0$ be a smooth  plane curve of degree $d$ defined over $\kk$, and let  $\kk(x,y)$ be its function field.  Assume  that $x$ is a separating variable of $\kk(x,y)$. If $(0,1,\epsilon)$ is the order sequence of  $\ff$, then the ramification divisor of $\ff$ is defined by
\begin{equation}\label{ramification}
R:=\divi\left(D^{(\epsilon)}_x(y) \right)+(1+\epsilon)\divi(dx)+3E,
\end{equation}
where $D^{(i)}_x$ is the $i$-th Hasse derivative with respect to $x$,  and $E=\sum\limits_{P \in \ff} e_P P$, with $e_P=-\min\{0,v_P(x),v_P(y)\}$, with $v_P$ denoting the discrete valuation at $P$. Note that $\deg(E)=d$ and $\deg(\divi(dx))=d(d-3)$, and then
\begin{equation}\label{degram}
\deg(R)=(1+\epsilon)d(d-3)+3d.
\end{equation}
For  any point  $P\in   \ff$,  let $T_{P}  \ff$ be  the   tangent line to $\ff$ at $P$.   If $j(P):=I(P, \ff  \cap T_{P}\ff)$  denotes the intersection multiplicity of $\ff$ and $T_{P}\ff$ at $P$,  then  it follows from   \cite[Theorem 1.5]{SV})  that   $v_P(R)\geq j(P)-\epsilon$, and equality holds if and only if $p \nmid { j(P) \choose \epsilon}$. In particular, $R$ is  an effective divisor.

Let $H$ be a subgroup of $\aut_{\kk}(\ff)$. The stabilizer of $P \in \ff$ in $H$ will be denoted by $H_P$, and the orbit of $P$ will be denoted by $H(P)$. If $t$ is a local parameter at $P$, the $i$-th ramification subgroup of $H$ at $P$ is
\begin{equation}\label{rg}
H_P^{(i)}=\{\sigma \in H_P \ | \ v_P(\sigma(t)-t) \geq i+1\}.
\end{equation}
 Here, $H_P=H_P^{(0)}\supseteq H_P^{(1)}\supseteq \cdots$, and $H_P^{(k)}=\{1\}$ for a sufficiently large $k$. Let $\ff/H$ denote the quotient curve of $\ff$ by $H$ and $g(\ff/H)$ denote its genus. The Riemann-Hurwitz formula gives
\begin{equation}\label{rh}
d(d-3)=|H|\big(2g(\ff/H)-2\big)+\sum_{P \in \ff}\sum_{i \geq 0}(|H_P^{(i)}|-1),
\end{equation}
see \cite[Theorem 11.72]{HKT}.

The following  important   results will be used in the proof of  Theorem \ref{ThmAut}.
\begin{thm}\label{Roq}(Roquette, \cite{Roq}) Let $\mathcal{X}$  be an  irreducible
curve of genus $g\geq 2$  defined over  a field of characteristic $p>g+1$.
Then  $|\Aut (\mathcal{X})|\leq 84(g_{\mathcal{X}}-1)$ holds,  except for  the hyperelliptic 
 curve $\mathbf{v}(Y^p-Y-X^2)$, with $g =\frac{1}{2}(p + 1)$ and $|\Aut (\mathcal{X})| = 2p(p^2-1)$.
\end{thm}

\begin{thm}\label{Henn}(Henn, \cite{Henn})  Let $\mathcal{X}$  be an  irreducible
curve of genus $g\geq 2$. If a subgroup $G$ of $\aut (\mathcal{X})$ has order at least
$8g^3$, then $\mathcal{X}$ is birationally equivalent to one of the following plane curves.

\begin{enumerate}[\rm(I)]
\item The hyperelliptic curve $\mathbf{v}(Y^2 + Y + X^{2^k + 1})$ with $p=2$, and $g=2^{k-1}$, $k\geq 2$,  $|\aut (\mathcal{X})|=2^{2k+1}(2^k+1)$,   $\aut (\mathcal{X})$ fixes a point
$P\in \mathcal{X}$.
\item The hyperelliptic curve $\mathbf{v}(Y^2 - (X^n-X))$ with $p>2$, $n=p^k>3$, $g=\frac{1}{2}(n-1)$, $\aut (\mathcal{X})/M\cong \rm{PGL}(2,n)$, $|M|=2$, $|\aut (\mathcal{X})|=2(n+1)n(n-1)$.
\item The Hermitian curve $\mathbf{v}(Y^n + Y - X^{n + 1})$ with $n=q^t \geq 2$ and $g=\frac{1}{2}(n^2-n)$, $\aut (\mathcal{X})\cong \rm{PGU}(3,n)$,  $|\aut (\mathcal{X})|=(n^3+1)n^3(n^2-1)$.
\item The DLS curve (the Deligne-Lusztig curve arising from the Suzuki group) $\mathbf{v}(X^{n_{0}}(X^n + X) - (Y^n + Y))$ with $p=2$, $n_{0}=2^r$, $r\geq 1$, $n=2n_{0}^2$ and $g=n_{0}(n-1)$,  $\aut (\mathcal{X})\cong Sz(n)$, where $Sz(n)$ is the Suzuki group, $|\aut (\mathcal{X})|=(n^2+1)n^2(n-1)$.
\end{enumerate}
\end{thm}

Hereafter,  we will focus on  the  smooth  curve  $\mathcal{H}_n:\, XY^n+YZ^n+X^nZ=0$  defined over $\F_{p}$, where   $n\geq 3$ and $p\nmid n^2-n+1$. 
Note that for  $\mathcal{H}_n$, equation \eqref{degram} reads
\begin{equation}\label{degramHn}
\deg(R)=(1+\epsilon)(n^2-n+1)+3(n-\epsilon).
\end{equation}

\section{Weierstrass points for the  morphism  of lines}

Let $R$ be the   ramification divisor of $\mathcal{H}_n$.  This section   provides  a complete  description of the points $P\in\supp(R)$ and   their orders  $j(P)$.

\begin{lem}\label{points} Let $\mathbb{K}(x,y)$  be    the function  field of  $\mathcal{H}_n$,  and  consider the set of points  $\{P_1,P_2,P_3\} \subseteq \mathcal{H}_n$. Then the following hold.
\begin{enumerate}[\rm(i)]
\item  For any point  $P=(a:b:c)  \in \mathcal{H}_n$,  $abc=0$ if  and only if $P\in \{P_1,P_2,P_3\}$.
\item $j(P_1)=j(P_2)=j(P_3)=n$.
\item  $\divi (x)=(n-1)P_2+P_3-nP_1$   and $\divi (y)=nP_3-(n-1)P_1-P_2$.
\item $y/x$, $y^{-1}$ and $x$ are local paramenters at $P_1$,  $P_2$, and $P_3$, respectively.
\end{enumerate}
\end{lem}
\begin{proof}
Consider the lines  $\ell_1:\,  Z=0$, $\ell_2:\, X=0$, and  $\ell_3:\, Y=0$. The divisors cut out on $\mathcal{H}_n$ by these lines are
$\ell_1\cdot \mathcal{H}_n=nP_1+P_2 $, $\ell_2\cdot \mathcal{H}_n=nP_2+P_3$,  and $\ell_3\cdot \mathcal{H}_n=P_1+nP_3$.
This proves  the  first three assertions.  Clearly, (iii) implies (iv).
\end{proof}

\begin{prop}\label{prop-dx} Let $\mathbb{K}(x,y)$ be    the function  field of  $\mathcal{H}_n$ and $dx$   be the differential of $x$.  Then
\begin{equation}\label{div-dx}
\divi (dx)= \begin{cases}
(n^2-2n)P_1+(n-2)P_2,  \text{ if } p\mid n.\\
-(n+1)P_1+(n^2-1)P_2,  \text{ if } p\mid n-1.\\
-(n+1)P_1+(n-2)P_2+\sum\limits_{i=1}^{n^2-n+1}Q_i,  \text{  if  } p\nmid n(n-1),
\end{cases}
\end{equation}
 where   $Q_i=(x_i: (\frac{n}{1-n})x_i^n  :1), \text{ and   } x_i^{n^2-n+1}=-\frac{(1-n)^{n-1}}{n^n}$.

%cujas respectivas retas  tangentes s\~ao dadas por $Z=0$, $X=0$ e $Y=0$
\end{prop}
\begin{proof} Consider  the curve  $\mathcal{C}:\, nXY^{n-1}+Z^n=0$, and let $\mathcal{C} \cdot  \mathcal{H}_n $ be   the corresponding  divisor cut out on $\mathcal{H}_n$. If $p\mid n$, then the computation of $\mathcal{C} \cdot  \mathcal{H}_n$ is trivial.   Otherwise, $\mathcal{C}$ is  a rational curve  and the parametrization $\phi:\mathbb{P}^1 \longrightarrow \mathcal{C}$ given by   $(u:v)\mapsto  (-u^n:nv^n:nuv^{n-1})$  gives
\begin{equation}\label{div-dx-aux}
 \mathcal{C} \cdot  \mathcal{H}_n=\begin{cases}
n^2P_1+nP_2,  \text{ if } p\mid n.\\
(n-1)P_1+(n^2+1)P_2,  \text{ if } p\mid n-1.\\
(n-1)P_1+nP_2+\sum\limits_{i=1}^{n^2-n+1}Q_i,  \text{  if  } p\nmid n(n-1),
\end{cases}
\end{equation}
where   $Q_i=(x_i: (\frac{n}{1-n})x_i^n  :1), \text{ and   } x_i^{n^2-n+1}=-\frac{(1-n)^{n-1}}{n^n}$. Now direct computation using Lemma \ref{points}, the divisor  \eqref{div-dx-aux}, and  $dx=-(\frac{nxy^{n-1}+1}{y^n+nx^{n-1}})dy$ proves   \eqref{div-dx}.
\end{proof}

\begin{cor}\label{cor-bridge} For the curve  $ \mathcal{H}_n$, the ramification  divisor in  \eqref{ramification} is given by
\begin{equation}
R=\divi\left((nxy^{n-1}+1)^{\epsilon+1}D^{(\epsilon)}_x(y) \right)+\Big((\epsilon + 1)n^2+(1 - 2\epsilon)n - 3\Big)P_1+\Big((\epsilon +1)n - 2\epsilon + 1\Big)P_2.
\end{equation}
\end{cor}

\begin{proof}
Lemma \ref{points}(iii) gives the divisor $E=(n-1)P_1+P_2$.  Combining   Proposition  \ref{prop-dx} and equation   \eqref{div-dx-aux} on its proof, we obtain
$\divi (dx)=\divi (nxy^{n-1}+1)+(n^2-2n)P_1+(n-2)P_2$.  The result follows after substituting $E$ and $\divi (dx)$ in   \eqref{ramification}.
\end{proof}

      \begin{lem}\label{Hasse} Let $\mathbb{K}(x,y)$  be the function field of $\mathcal{H}_n$, where $f(x,y)=xy^n+y+x^n=0$.  If  
      $D_x^{(i)}y$ denotes the $i$-th Hasse derivative of $y$ with respect to $x$, then

\begin{numcases}{(nxy^{n-1}+1)^{\epsilon+1}  D_x^{(\epsilon)}y=}
mx( y^{n(p^r +1)-p^r} -x^{n-p^r+1}),  \text{ if }\epsilon= p^r=n/m,  \text{ with  } r\geq 1 \text{ and  } p\nmid m. \label{Lem1}\\
y^{n-1}(y^n+x^{n-1})(xy^{n-1}+1),  \text{ if }\epsilon= p=2 \text{ and  } n\equiv 1 \mod 4.\label{Lem2}\\
x(y^{3n-2}+x^{n-3}),  \text{ if }\epsilon= p=2 \text{ and  } n\equiv 3 \mod 4.\label{Lem3}\\
 f_xf_yf_{xy}-(f_x^2f_{yy}+f_y^2f_{xx})/2,  \text{ if }\epsilon=  2\neq p \text{ and  } p \nmid n.\label{Lem4}
\end{numcases}

\end{lem}
      \begin{proof}
      From $xy^n+y+x^n=0$, we have  that  $D_x^{(1)}y=-\frac{y^n+nx^{n-1}}{nxy^{n-1}+1}$. Now the higher-order derivatives will follow from $D_x^{(1)}y$ and the standard computations    using the basic properties of Hasse derivative (see e.g.  \cite[ Section 5.10]{HKT}).
      \end{proof}

\begin{lem}\label{factorization-g} For $p>2$ and  $(n^2-n)(n^2-n+1) \not \equiv 0 \mod p$, let $g(T) \in \F_{p}[T]$ be the polynomial 
\begin{equation}\label{polyg}
(n-1)T^3+ (n^3 - 3n^2 + 6n - 2)T^2+ (n^3 - 3n^2 + 3n + 1)T- (n-1).
   \end{equation}

 Then $g(T)$ has  discriminant $\Delta=(n^2-n+1)^4(n^2-4n+7)^2$,
 and
 
\begin{numcases}{g(T)=}
(n-1)(T-\alpha)^3,  \text{ if } p|(n^2-4n+7)\label{pdivide},\\
(n-1)(T-\eta)(T+1/(\eta+1))(T+(\eta+1)/\eta), \text{ if } p\nmid n^2-4n+7\label{pndivide},
\end{numcases}
 where $\alpha \in \F_p$ is a primitive cubic root of unity, and  $\eta\in \kk$ is any root of $g(T)$.
 In particular,   $p\geq 7$  for  condition \eqref{pdivide}.
\end{lem}

\begin{proof}
The discriminant $\Delta$ is obtained by standard computation. Note
that since $p>2$ and  $p\nmid n^2-n+1$,  the condition $p\mid n^2-4n+7$ implies $p>3$.
Thus, if $\alpha \in \kk$ is a primitive cubic root of  unity, then  $1-2\alpha$  e   $1-2\alpha^2$ are the roots of $n^2-4n+7\equiv 0 \mod p$, and then     \eqref{polyg} implies  \eqref{pdivide}. Also, since
$n\in\mathbb{Z}$, it follows that $\alpha \in  \F_p$, and then $p\equiv 1 \mod 3$ gives $p\geq 7$.
For \eqref{pndivide},     one can check that  $g(0)=-g(1)=n-1 \not \equiv 0 \mod p$ and the  identities  
$$-(T+1)^3g(-1/(T+1))=g(T)=T^3g(-(T+1)/T).$$
\end{proof}

\begin{thm}\label{TeoWei}
Let  $R$   be   the ramification divisor of the smooth Hurwitz curve $\mathcal{H}_n$ defined over $\F_{p}$.
Then 
$$\supp(R)=\{P_1,P_2,P_3\} \cup \mathcal{W},$$
where $\mathcal{W}$ is characterized as follows. 

\begin{enumerate}[\rm(1)]

\item  If $p\mid n$, then
\begin{equation}\label{Wnclassico}
\mathcal{W}=\{(\lambda: t^n:t^{n-1})\,  \, |  \, \, t^{n^2-n+1}=\lambda^{n-p^r-1} \text{ and } \lambda^{p^r+1}+\lambda+1=0 \}.
\end{equation}
where $n=p^rm$, with $r\geq1$ and $p\nmid m$.  In particular,  $j(P)=p^r+1$ for all
 $P \in \mathcal{W}$.

\item   If $p\nmid n$, then

\vspace{0.3 cm}

 \begin{enumerate}[\rm(i)]
 
\item Case $p=2$. If $n\equiv 1 \mod 4$,  then $\mathcal{W}=\emptyset$. Otherwise,  
\begin{equation}\label{Wj=3 car2}
\mathcal{W}=\{(\lambda:  t^n:t^{n-1})\, \,| \, \,  t^{n^2-n+1}=\lambda^{n-3} \text{ and } \lambda^3+ \lambda+1=0 \},
\end{equation}
 and $j(P)=3$ for all $P\in \mathcal{W}$.
 
\item Case $p>2$. If $p\mid n-1$, then $\mathcal{W}=\emptyset$. Otherwise,  for the polynomial  $g$ given in Lemma \ref{factorization-g}, we have

\begin{enumerate}[\rm(a)]
\item if $p\mid(n^2-4n+7)$, then 
\begin{equation}\label{Wj=5}
\mathcal{W}=\{(t: \lambda t^n:1)\,  \, |  \, \, t^{n^2-n+1}=\lambda^{2(n+1)} \text{ and }  g(\lambda)=0 \},
\end{equation}
and $j(P)=5$ for all $P\in \mathcal{W}$.
\item if $p\nmid (n^2-4n+7)$, then 
\begin{equation}\label{Wj=3}
\mathcal{W}=\{(t: \lambda t^n:1)\, \,| \, \,  t^{n^2-n+1}=-(\lambda+1)\lambda^{-n} \text{ and } g(\lambda)=0 \},
\end{equation}
 and $j(P)=3$ for all $P\in \mathcal{W}$.
\end{enumerate}

\end{enumerate}

\end{enumerate}

\end{thm}

\begin{proof}

Set $x=X/Z$ and $y=Y/Z$, and let $\mathbb{K}(x,y)$ be the function field of $\mathcal{H}_n$. 
From Corollary \ref{cor-bridge},   the points $P \in \mathcal{W}$ can be obtained by intersecting the affine  curve $f(x,y):=xy^n+x^n+y=0$
with the one associated to $(nxy^{n-1}+1)^\epsilon  D_x^{(\epsilon)}y$   given in Lemma \ref{Hasse}. 

\begin{enumerate}[\rm(1)]
\item Case  $p\mid n$. In this case,  
\eqref{Lem1}  in Lemma \ref{Hasse} yields the curve $y^{n(p^r +1)-p^r} =x^{n-p^r+1}$, and a simple calculation
shows that the intersection points are those of $f(x,y)=0$ subjected to
$$(xy^{n-1})^{p^r+1}+xy^{n-1}+1=0.$$
Thus  for  any root $\lambda$ of the separable polynomial  $T^{p^r+1}+T+1$, we have  $n^2-n+1$ intersection points $P=(x,y)$, where $xy^{n-1}=\lambda$, and $y$ is given by solving $f(\frac{\lambda}{y^{n-1}},y)=0$.
This proves \eqref{Wnclassico}. In addition, since $\#\mathcal{W}=(p^r+1)n^2-n+1$ and $j(P)\geq p^r+1$ for all
 $P \in \mathcal{W}$, equation  \eqref{degramHn} implies  $j(P)= p^r+1$ for all
 $P \in \mathcal{W}$.
 
 \item  Case $p\nmid n$. If $p=2$ and  $n\equiv 1 \mod 4$, then \eqref{Lem2}  in Lemma \ref{Hasse} yields the curve $(y^n+x^{n-1})(xy^{n-1}+1)=0$.
Since  this curve intersects $xy^n+x^n+y=0$ only at points $P=(x,y)$ for which $xy=0$, it follows that   $\mathcal{W}=\emptyset$.
 For $p=2$ and  $n\equiv 3 \mod 4$, the proof is similar to the case   $p\nmid n$.  For $p>2$ and   $n\equiv 1 \mod p$, note that
 $f_xf_yf_{xy}-(f_x^2f_{yy}+f_y^2f_{xx})/2=y^{n-1}(xy^{n-1}+1)(x^{n-1}+y^n)$, and analogous  to the case $p=2$ and  $n\equiv 1 \mod 4$, we have  $\mathcal{W}=\emptyset$.
 
 Next, we assume $p\nmid (n-1)$. Direct computation shows that the problem of intersecting the curves $f_xf_yf_{xy}-(f_x^2f_{yy}+f_y^2f_{xx})/2=0$ and $f(x,y)=0$  can be  reduced to  that of intersecting  $f(x,y)=0$ with the curve associated to 
   \begin{equation}\label{polyh}
   h(x,y)=(n-1)y^3+ (n^3 - 3n^2 + 6n - 2)x^ny^2+ (n^3 - 3n^2 + 3n + 1)x^{2n}y- (n-1)x^{3n}.
   \end{equation}
\begin{enumerate}[\rm(a)]

\item For $p\mid n^2-4n+7$, equation \eqref{pdivide}
gives $h(x,y)=(n-1)(y-\alpha x^n)^3$. From $f(x,\alpha x^n)=\alpha^n x^{n^2+1}+\alpha x^n+x^n$, we arrive at  the $n^2-n+1$  intersection  points $(t, \alpha t^n)$, where $t$ are  roots of 
\begin{equation}
x^{n^2-n+1}=\alpha^{2(n+1)},
\end{equation}
which proves  \eqref{Wj=5}. Note that since  $h(x,y)=(n-1)(y-\alpha x^n)^3$,  the curves
$f(x,y)=0$ and $h(x,y)=0$ intersect  at each $P\in \mathcal{W}$ with multiplicity at least $3$. That is, $v_P(R)\geq 3$ for all $P\in \mathcal{W}$. Since $\#\mathcal{W}=n^2-n+1$, equation  \eqref{degram} implies $v_P(R)=3$ for all $P\in \mathcal{W}$. Therefore, since $p\geq 7$, we have $v_P(R)=j(P)-2$ and then $j(P)=5$.

\item For $p\nmid n^2-4n+7$,  if $\lambda$ is any of the three distinct roots of $g(T)$, then the corresponding factor  $y-\lambda x^n$ of $h(x,y)$  yields  intersection points $(t, \lambda t^n)$, where

\begin{equation}
t^{n^2-n+1}+(\lambda+1)\lambda^{-n}=0,
\end{equation}
which proves  \eqref{Wj=3}.  As in the previous case, a counting argument  gives $j(P)=3$ for all $P\in \mathcal{W}$.

\end{enumerate}

\end{enumerate}

\end{proof}

\begin{cor}\label{OmegaAction}
Consider the  smooth Hurwitz curve $\mathcal{H}_n: XY^n+YZ^n+X^nZ=0$ defined over $\F_{p}$. If  $n>3$ is not a power of $p$, then 
 for any point $P\in \mathcal{H}_n$, we have $ j(P) = n$ if and only if $P \in \{P_1,P_2,P_3\}$. In particular,  $\Aut_{\mathbb{K}}(\mathcal{H}_n)$ acts
on the set $\{P_1,P_2,P_3\}$.
\end{cor}

\begin{proof}

From Lemma \ref{points} (ii),  we have  $ j(P_i) = n$  for $i=1,2,3$. For the remaining points $P\in \mathcal{H}_n$, we have that either  $ j(P) = \epsilon<n$ or 
$P \in \mathcal{W}$, where $\mathcal{W}$  is completely characterized by Theorem \ref{TeoWei}. The last assertion follows from the fact that
$\Aut_{\mathbb{K}}(\mathcal{H}_n) \leqslant \rm{PGL}(3,\mathbb{K})$, as $\mathcal{H}_n \subseteq \mathbb{P}^2$ is smooth (see  e.g. \cite{Chang}).

\end{proof}

\section{The automorphism group  of   $\mathcal{H}_n$}

Let us recall that     $n\geq 3$ is   such that  $p\nmid n^2-n+1$, that  is, the Hurwitz curve  $\mathcal{H}_n$   over $\kk$  is smooth.
Hereafter,    $\xi \in  \kk$ denotes  a primitive  $(n^2-n+1)$-th root of unity.

\begin{lem}\label{aux}   If      $\sigma$ and   $\mu$  are   the projective  transformations   associated to   the matrices
\begin{equation}
\begin{pmatrix}
       \xi^{-n+1}      &  0  &   0 \\
      0      &  \xi  &  0\\
      0       &  0  &   1
\end{pmatrix}
\text{ and }
\begin{pmatrix}
    0       &  1  &   0 \\
      0       &  0  &  1 \\
    1      &  0  &   0
    \end{pmatrix},
\end{equation}
respectively, then the  following  hold. 
\begin{enumerate}[\rm(i)]
\item $\langle  \mu  \rangle$  and  $\langle  \sigma  \rangle$  are subgroups of $\Aut_\kk(\mathcal{H}_n)$ of  order $3$ and $n^2-n+1$, respectively.
\item $\langle  \sigma  \rangle  \cap \langle  \mu  \rangle =\{\bf{1}\}$.
\item $\mu \sigma  \mu^{-1}=\sigma^{n-1}.$
\end{enumerate}
\begin{proof}
The three   assertions follow  from straightforward computations.
\end{proof}
\end{lem}

The next result presents the   automorphism group of the smooth  Hurwitz curve $\mathcal{H}_n$  defined over $\F_{p}$. The particular cases
$n\in \{3, p^r\}$  are well known, but we provide   them here for the sake of completeness. Also, the case $n = p^r+1$ has been recently settled in \cite{GD}.

\begin{thm}\label{ThmAut}

 If $\Aut_\kk(\mathcal{H}_n)$ denotes the full automorphism group of the smooth  Hurwitz curve
 $$\mathcal{H}_n:\,  XY^n+YZ^n+X^nZ=0$$  defined over $\F_{p}$, then

\begin{equation}
\Aut_\kk(\mathcal{H}_n)  \cong   
\begin{cases}
{\rm PGU}(3, n),  \text{ if } n  \text{  is a power  of  }  p, \\
{\rm PSL}(2, 7),  \text{ if } n=3\neq p,  \\
C_{n^2-n+1} \rtimes_{\varphi} C_3,  \text{  otherwise,  }
\end{cases}
\end{equation}

where $\varphi: C_3=\langle g \rangle  \longrightarrow \mathbb{U}_{n^2-n+1}$ is given by $g \mapsto n-1$.
 
\end{thm}
\begin{proof}

For  $n=p^r$, the  result is well known, as   $\mathcal{H}_n$ is  isomorphic to the  Hermitian curve  (see  e.g.   \cite[ Remark 8.19]{HKT}).
Let us consider the case  $n=3 \neq p$.   For $p=2$, note that  the  determinants

\begin{equation}\label{dets}
D_1=\left|\begin{array}{ccc}
X & X^{2} & X^{8} \\
Y & Y^{2} & Y^{8} \\ 
Z & Z^{2} & Z^{8} 
\end{array} \right| \  \mbox{ and } \ 
D_2=\left|\begin{array}{ccc}
X & X^{2} & X^{4} \\
Y & Y^{2} & Y^{4} \\ 
Z & Z^{2} & Z^{4} 
\end{array} \right|
\end{equation} 
are such that $D_1/D_2$ is a polynomial of degree 4   giving  rise to the  smooth  curve
$$\mathcal{C}:\, (X+Y+Z)^4+(XY+YZ+XZ)^2+XYZ(X+Y+Z)=0.$$
Thus it follows from  elementary properties of determinants   that  the whole of  $\rm{PGL}(3,\mathbb{F}_2)$  is  a subgroup of  $\Aut (\mathcal{C})$. Moreover, 
if  $\zeta$ is a generator of the cyclic group $\F_{8}^{\times}$, then one can check  that
$$(X:Y:Z)\mapsto (X+Y+Z:  \zeta^2 X+\zeta^4 Y+\zeta  Z:  \zeta X+\zeta ^2Y+\zeta^4Z )$$
is an  isomorphism  from  the Klein quartic $\mathcal{H}_3$  to  curve  $\mathcal{C}$. In particular, $\rm{PSL}(3,\mathbb{F}_2)  \hookrightarrow \Aut_{\kk} (\mathcal{H}_3)$, and
then  $\#\Aut_{\kk} (\mathcal{H}_3)=168 m$ for some integer $m\geq 1$.  If $m\geq 2$, then $\#\Aut_{\kk} (\mathcal{H}_3)\geq  336> 8g^3$   contradicts Theorem  \ref{Henn}.
Therefore, $\Aut_{\kk} (\mathcal{H}_3)\cong \rm{PSL}(3,\mathbb{F}_2)$.  For $p>3$, we  have   $p\neq 7$ (as $\mathcal{H}_3$ is nonsingular)  and then   Theorem  \ref{Roq}   implies   $\#\Aut_{\kk} (\mathcal{H}_3)\leq 168$.    Therefore,  the classical    argument for  zero characteristic can  be used, and it follows that 
\begin{equation}
\Aut_{\kk} (\mathcal{H}_3)=\langle \sigma,\mu,T  \rangle \cong \rm{PSL}(3,\mathbb{F}_2),
\end{equation}
where  $\sigma$, $\mu$  are given  by  Lemma \ref{aux},  and $T$ is   the projective  transformation   associated to   the matrix
\begin{equation}
\begin{pmatrix}
    \xi^3-\xi^2        &  \xi-\xi^4  &   1-\xi^5 \\
       \xi-\xi^4       &  1-\xi^5 &   \xi^3-\xi^2   \\
    1-\xi^5       &  \xi^3-\xi^2  &   \xi-\xi^4
\end{pmatrix},
\end{equation}
and   $\xi$  is     a  primitive seventh  root of unity (see e.g.   \cite[ Section  6.5.3]{Dolga}).

Now let us assume $n>3$.  By Corollary \ref{OmegaAction}, $\Aut_\kk(\mathcal{H}_n)$ admits a   permutation representation  $\rho:\, \Aut_\kk(\mathcal{H}_n) \longrightarrow  \sym (\Omega)$. 
Note that  $\ker \rho=\{\vartheta \in \Aut_\kk(\mathcal{H}_n):\,  \vartheta(P_i)=P_i,  \text{ for }  i=1,2,3\}$ is  the set of  maps  $(X:Y:Z) \mapsto (\alpha X: \beta Y:Z)$,  where
 $\alpha, \beta  \in \kk \backslash  \{0\}$ are subject to
\begin{equation}\label{relation}
(\alpha \beta ^n)XY^n+\beta YZ^n+\alpha^nZX^n=\gamma(XY^n+YZ^n+ZX^n)
\end{equation}
for some $ \gamma \in \kk$.   This gives   $\alpha=\beta^{-n+1}$ and $\beta^{n^2-n+1}=1$,  and then $\ker \rho=\langle \sigma \rangle \trianglelefteq \Aut_\kk(\mathcal{H}_n)$, where $\sigma$  is given by Lemma \ref{aux}.  Since  $\mu$ intersects   $\ker \rho$  trivially, it follows  that  $3\leq |\im  \rho |\leq 6$. On the  other hand,  there is no $\varphi \in \Aut_\kk(\mathcal{H}_n)$ such that $\varphi(P_3)=P_3$ and $\varphi(P_1)=P_2$. In fact, one can check that any such a  $\varphi$  should  be  of type  $(X:Y:Z) \mapsto (\alpha Y: \beta X:Z)$,  with   $\alpha, \beta  \in \kk \backslash  \{0\}$ subject to \eqref{relation}, which  contradicts   $\varphi  \notin  \ker \rho$.  Therefore, $|\Aut_\kk(\mathcal{H}_n) / \langle \sigma \rangle|=|\im  \rho |= | \langle  \mu  \rangle|=3$, and the result follows  from  %$\Aut_\kk(\mathcal{H}_n)=\langle  \sigma  \rangle  \cdot \langle  \mu  \rangle$  and 
Lemma \ref{aux}.
\end{proof}

%\begin{rem}
%Let $n \neq 3$ and $n \neq p^r, r \geq 1$. Then the curve $\cH_n$ is a Galois subcover of the Fermat curve $\cF_{n^2-n+1}: U^{n^2-n+1}+V^{n^2-n+1}+1 = 0$ by the subgroup of $\aut(\cF_{n^2-n+1})$ generated by $\sigma(U,V) = (\lambda^nU,\lambda^{n-1}V)$, where $\lambda$ is a primitive $n^2-n+1$-th root of the unity. Let $K(\cF_{n^2-n+1}) = K(u,v)$. It is easily seen that 

%\end{rem}

\section{Galois subcovers of $\mathcal{H}_n$}

In several situations, the construction of quotient curves of a given curve is desirable. To this end, one must know the stabilizers of all points of the curve. Moreover, if the order of the stabilizer of a given point is divisible by $p$ (i.e., the stabilizer is nontame), then the ramification groups of such point must be computed. In this section, we describe all subgroups of $G=\aut_{\kk}(\mathcal{H}_n)$ up to conjugacy and all points of $\mathcal{H}_n$ with nontrivial stabilizers, together with their respective stabilizers.  For the nontame cases, the ramification groups are also computed. As a consequence, we obtain the complete spectrum of the genera of quotient curves of the Hurwitz curve.  

In what follows, we establish the following notation:
\begin{itemize}
\item $S_d:=\langle \sigma^{\frac{n^2-n+1}{d}} \rangle$, where $d$ divides $n^2-n+1$.
\item $\tau:=\sigma^{\frac{n^2-n+1}{3}}$ and $T_0:=\langle \tau \rangle$ if $n \equiv 2 \mod 3$.
\item $T_i:=\langle \mu \sigma^i\rangle$, for $i=1,\ldots,n^2-n+1$.
\end{itemize}

We start with the classification of the subgroups of $G$.

\begin{prop}\label{subgroups}
The subgroups $H\leq G$ are the following.
\begin{itemize}
\item[(a)] If $|H|=3$, then
\begin{enumerate}
\item  $H=T_0 \subset Z(G)$, if $n\equiv 2 \mod 3$.
   \item  $H=T_i$,  for  $i=1,\ldots,n^2-n+1$, with such groups forming a single conjugacy class of size $n^2-n+1$ if $n\not\equiv 2 \mod 3$, and three conjugacy classes of size $\frac{n^2-n+1}{3}$, represented by $T_{j\frac{n^2-n+1}{3}}$ with $j=1,2,3$, otherwise. 
\end{enumerate}
\item[(b)] If $|H|=d$, where $d|(n^2-n+1)$, then
\begin{enumerate}
   \item  $H=T_0$ or $T_i$, if $d=3$, with the conjugacy classes described in (a).
   \item   $H=S_d \triangleleft G$, for the other cases.
   \end{enumerate}
\item[(c)] If $|H|=3d$, where $d|(n^2-n+1)$, then
$H=T_i \cdot S_d$,  for $i=0, 1,\cdots,n^2-n+1$. If $n\not\equiv 2 \mod 3$, all such groups are conjugated. If $n\equiv 2 \mod 3$, then 
\begin{enumerate}
\item If $i=0$, we have one conjugacy class with a single group, namely $S_{3d}$ (in this case, $3 \nmid d$).
\item If $i>0$ and $3 \nmid d$, there are three conjugacy classes of size $\frac{n^2-n+1}{3}$, represented by $T_{j\frac{n^2-n+1}{3}} \cdot S_d$, with $j=1,2,3$.  If $i>0$ and $3 | d$, then all groups are conjugated.
\end{enumerate}
\end{itemize} 
\end{prop}
\begin{proof}The list of subgroups of $G$ follows by straightforward computations, using the fact that $G=\langle\sigma,\mu \rangle$ is such that $\sigma^{n^2-n+1}=\mu^3=1$ and $\mu \sigma \mu^{-1}=\sigma^{n-1}$.
\begin{itemize}
\item[(a)] Assume $|H|=3$ and suppose $n \not\equiv 2 \mod 3$. Then  $n^2-n+1 \not\equiv 0 \mod 3$. Thus, by the Sylow Theorem, the Sylow $3$-subgroups of $G$ have order $3$ and are all conjugated. Hence, such subgroups are  generated by $\sigma^k \mu \sigma^{-k}$ for $k \in\{0,\ldots,n^2-n\}$, as $G=\langle \sigma\rangle \rtimes  \langle \mu \rangle$. We have $\sigma^k \mu \sigma^{-k} \neq \sigma^l \mu \sigma^{-l}$ for all $k,l \in \{0,\ldots,n^2-n\}$ such that $k \neq l$. Indeed, if $\sigma^k \mu \sigma^{-k}=\sigma^l \mu \sigma^{-l}$ for $k,l \in \{0,\ldots,n^2-n\}$ with $k>l$, we obtain $\sigma^d \mu=\mu \sigma^{d}$, where $d=k-l$. This gives
$$
\left(
\begin{array}{ccc}
0&\xi^{d(-n+1)}&0\\
0&0&\xi^d\\
1&0&0
\end{array}\right)=
\left(
\begin{array}{ccc}
0&\xi^{dn}&0\\
0&0&\xi^{d(n-1)}\\
1&0&0
\end{array}\right).
$$
Thus $n^2-n+1 | d(2n-1)$ and $n^2-n+1 | d(n-2)$. Since $d<n^2-n+1$, we conclude that $n^2-n+1$, $2n-1$ and $n-2$ have a common factor $\ell$. Then $\ell | 2n-1$ and $\ell | 2n-4$ gives $\ell=3$, which contradicts  $n \not\equiv 2 \mod 3$. A straightforward computation also shows that $\sigma^d \mu \neq\mu^2 \sigma^{d}$ for all $d \in \{0,\ldots,n^2-n\}$ (this does not depend on the congruence of $n$ modulo $3$). Therefore $\langle \sigma^i \mu \sigma^{-i} \rangle \neq \langle \sigma^j \mu \sigma^{-j} \rangle$ if $i \neq j$.

Assume now $n \equiv 2 \mod 3$. We will show that, up to conjugacy, $H \in\{T_0,\langle \tau^j\mu \rangle \ | \ j=0,1,2\}$, with $T_0 \leq Z(G)$, and the conjugacy class of $\langle \tau^j\mu \rangle$ has size $\frac{n^2-n+1}{3}$ for $j=0,1,2$. Recall that $n \equiv 2 \mod 3$ is equivalent to $n^2-n+1 \equiv 0 \mod 3$. Since $\frac{n^2-n+1}{3}$ is not divisible by $3$, we conclude that the Sylow $3$-subgroups of $G$ have order $9$. The element $\alpha:=\xi^{\frac{n^2-n+1}{3}}$ is a primitive cubic root of $1$, and then $\tau:(X:Y:Z) \mapsto (\alpha^2 X: \alpha Y: Z)$. One can check that $\tau$ commutes with $\mu$, the group $K=\langle \tau, \mu \rangle$ has order $9$ (in particular, it is a Sylow $3$-subgroup of $G$) and every element of $K$ has order $3$. Hence, every element of order $3$ of $G$ is conjugated to some element of $K$. Therefore, since $G=\langle \sigma\rangle \rtimes  \langle \mu \rangle$, an element of order $3$ of $G$ is of the form $\rho \theta \rho^{-1}$, where $\theta \in K$ and  $\rho=\sigma^k\mu^s \in G$,  with $k \in \{0,\ldots,n^2-n\}$ and $s \in \{0,1,2\}$. Since $\mu^s \in K$, we obtain $\rho \theta \rho^{-1}=\sigma^{k} \tilde{\theta} \sigma^{-k}$, with 
$
\tilde{\theta} \in K 
$.

We clearly have $\sigma^{k}\mu \sigma^{-k}=\mu$ for $k=\frac{n^2-n+1}{3}$. Assume that this equality holds for some $0<k<\frac{n^2-n+1}{3}$. Then, a computation as in case $n \not\equiv 2 \mod 3$ shows that $n^2-n+1$ divides both $k(n-2)$ and $k(2n-1)$, and $\gcd(k(n-2),k(2n-1))=3k$. Thus $n^2-n+1$ divides $3k$, a contradiction. Moreover, as we saw previously, $\sigma^d \mu \neq\mu^2 \sigma^{d}$ for all $d \in \{0,\ldots,n^2-n\}$.  Hence, the subgroups  $\langle \sigma^{k} \mu \sigma^{-k} \rangle$, where $k \in\{1,\ldots,\frac{n^2-n+1}{3}\}$, are pairwise distinct. Since $\tau$ is in the center of $G$, $\langle \tau \rangle$ is the only group in its conjugacy class. The fact that $\tau$ is central in $G$ also gives $\sigma^{k} \tau \mu \sigma^{-k}= \tau\mu$ for $k=\frac{n^2-n+1}{3}$, $\sigma^{k}\tau \mu \sigma^{-k}\neq \tau\mu$ for $k \in\{1,\ldots,\frac{n^2-n+1}{3}-1\}$  and $\sigma^{d}\tau \mu \sigma^{-d}\neq \tau^2\mu^2$ for $d \in\{0,\ldots,n^2-n\}$. Thus $\langle \sigma^{k} \tau \mu \sigma^{-k} \rangle$, where $k \in\{1,\ldots,\frac{n^2-n+1}{3}\}$, are pairwise distinct. An analogous argument shows that the same holds for $\langle\sigma^{k} \tau^2 \mu \sigma^{-k} \rangle$. Finally, since there are precisely $2(n^2-n+1)$ elements of order $3$ outside $\langle \sigma \rangle$ in $G$, the subgroups  $T_{j\frac{n^2-n+1}{3}}$  and  $T_{k\frac{n^2-n+1}{3}}$ are not conjugated if $j \neq k$, where $j,k \in\{1,2,3\}$. 
\item[(b)] Suppose $|H|=d$, where $d|(n^2-n+1)$. In view of (a), let us assume $d \neq 3$. Using item (a) and a counting argument, we conclude that $H \leq \langle \sigma \rangle$, and the result follows.
\item[(c)] Suppose $|H|=3d$, where $d|(n^2-n+1)$, i.e., $H=T_i \cdot S_d$. If $n \not\equiv 2 \mod 3$, then (a) and (b) imply that $H$ is conjugated to $T_{n^2-n+1}\cdot S_d$. 

Assume  $n \equiv 2 \mod 3$. If $i=0$, then $H=S_{3d}$, and the result follows from (b). Assume $i>0$ and $3 \nmid d$. In particular, $\tau \notin H$. By (a) and the equality $\sigma^{l}\mu\sigma^{-l}=\mu\sigma^{-l(n+1)}$ for all integer $l$, there exists  $j\in \{1,\ldots,\frac{n^2-n+1}{d}\}$ such that 
$$
		\mu\sigma^{i}\in\{\sigma^j \mu \sigma^{-j}, \sigma^j \mu\tau \sigma^{-j}, \sigma^j \mu\tau^{2} \sigma^{-j}\}.
$$
Then, given $(\mu\sigma^{i})^s\sigma^{k\frac{n^2-n+1}{d}} \in H$, where $s\in\{1,2\}$ and $k \in \{1,\ldots,\frac{n^2-n+1}{d}\}$, we have  
$$
(\mu\sigma^{i})^s\sigma^{k\frac{n^2-n+1}{d}} \in \{\sigma^j (\mu^s\sigma^{k\frac{n^2-n+1}{d}}) \sigma^{-j}, \sigma^j (\mu^s\tau^s\sigma^{k\frac{n^2-n+1}{d}}) \sigma^{-j}, \sigma^j (\mu^s\tau^{2s}\sigma^{k\frac{n^2-n+1}{d}}) \sigma^{-j}\}.
$$
Therefore, $H$ is conjugated to one of the following: $\langle \mu \rangle \cdot S_d$, $\langle \mu \tau\rangle \cdot S_d$ or $\langle \mu\tau^2 \rangle \cdot S_d$. Since these subgroups are not conjugated to each other, we have the conclusion. If $3|d$, then $ \tau \in H$. Hence $\sigma^j \mu \sigma^{-j} \in H$ for some $j$, which implies that $H$ is conjugated to $\langle \mu \rangle \cdot  S_d$. 
\end{itemize}
\end{proof}

Now that we have the classification of all subgroups of $G$ up to conjugacy, we want to explore which ones of them fix points of $\mathcal{H}_n$. Recall that $\Omega=\{P_1,P_2,P_3\} \subset \mathcal{H}_n$ is the fundamental triangle.

\begin{lem}\label{pfc}
The automorphism $\sigma$ fixes $\Omega$ pointwise and the remaining points of $\mathcal{H}_n$ are in long orbits of $\sigma$. Furthermore, no automorphism of $\mathcal{H}_n$ outside $\langle \sigma \rangle$ fixes a point of $\Omega$.
\end{lem}
\begin{proof}
The proof of the first claim is straightforward. For the second claim, note that if $\pi \in \aut(\mathcal{H}_n)\backslash \langle \sigma \rangle$, then $\pi=\mu^s\sigma^k$ for some $s\in\{1,2\}$ and $k \in \{0,1,\ldots,n^2-n\}$. Since
$$
\mu\sigma^k=\left(
\begin{array}{ccc}
0&\xi^{k}&0\\
0&0&1\\
\xi^{kn^2}&0&0
\end{array}\right)
 \ \ \ \ \text{ and } \ \ \ \mu^2\sigma^k=\left(
\begin{array}{ccc}
0&0&1\\
\xi^{kn^2}&0&0\\
0&\xi^k&0
\end{array}\right),
$$
the conclusion follows directly.
\end{proof}

\begin{lem}\label{ost}
Let $H$ be a nontrivial subgroup of $G$ that fixes some point $P \in \mathcal{H}_n$. Then either $H \leq\langle \sigma \rangle$ or $|H|=3$
\end{lem}
\begin{proof}
Suppose that $H \not\leq \langle \sigma \rangle$. Then $P \notin \Omega$ by Lemma \ref{pfc}. The Orbit-Stabilizer Theorem gives $|G_P||G(P)|=3(n^2-n+1)$. Again by Lemma \ref{pfc}, we have $|G(P)| \geq n^2-n+1$. Since $H \subset G_P$ and $n^2-n+1$ is odd, we obtain the result.
\end{proof}

\begin{prop}\label{pf3}
Assume $n \not\equiv 2 \mod 3$. The fixed points of $\langle \mu \rangle$ is $(1:1:1)$  if $p=3$ or are 
$
( \alpha:\alpha^2:1)$  and $( \alpha^2:\alpha:1)
$ 
if $p \neq 3$, where $\alpha$ is a primitive cubic root of $1$. Furthermore, besides the subgroups of $\langle \sigma \rangle$ and the conjugated of $\langle \mu \rangle$, no other subgroup of $G$ fixes a point of $\mathcal{H}_n$.
\end{prop}
\begin{proof}
Let $P=(a:b:c) \in \mathcal{H}_n$ fixed by $\mu$. Then $abc \neq 0$,  $(a/c)^3=1$ and $b/c=(a/c)^2$ (note that $(1:1:1) \in \mathcal{H}_n$ if and only if $p=3$).  The last statement follows from Proposition \ref{subgroups} item (a) and Lemma \ref{ost}.
\end{proof}

\begin{prop}\label{pf3-2}
Assume $n \equiv 2 \mod 3$. Then the following holds:
\begin{itemize}
\item[(1)] $\langle  \mu  \rangle$ has no fixed points.
\item [(2)]$\langle \tau \rangle$ fixes $\Omega$ pointwise.
\item[(3)] $\langle \tau \mu \rangle$ fixes pointwise the set
$
\{( \alpha:1:1), (1 :\alpha:1),(1 :1:\alpha) \},
$
where $\alpha$ is a primitive cubic root of $1$.
\item [(4)] $\langle \tau^2 \mu  \rangle$ fixes pointwise the set
$
\{( \alpha^2:1:1), (1 :\alpha^2:1),(1 :1:\alpha^2) \},
$
where $\alpha$ is a primitive cubic root of $1$. 
\end{itemize}
Moreover, besides the subgroups of $\langle \sigma \rangle$ and the conjugated of the groups described in (2), (3) and (4) above, no other subgroup of $G$ fixes a point of $\mathcal{H}_n$.
\end{prop}
\begin{proof}
A straightforward computation as in Proposition \ref{pf3} gives the result on the fixed points of the respective groups. Once again, the last statement follows from Proposition \ref{subgroups} item (a) and Lemma \ref{ost}.
\end{proof}

Since we are assuming that $\mathcal{H}_n$ is nonsingular (i.e., $p \nmid n^2-n+1$), we have that $\mathcal{H}_n$ has a $p$-group if and only if $p=3$ and $3 \nmid n^2-n+1$. As we saw in Proposition \ref{pf3-2}, $3 \nmid n^2-n+1$ is equivalent to $n \not\equiv 2 \mod 3$. In the next proposition, we determine the ramification groups of the points outside $\Omega=\{P_1,P_2,P_3\}$ with nontrivial stabilizer.

\begin{prop}\label{rg}
Assume that $p=3$ and $n \not\equiv 2 \mod 3$. If $P \in \mathcal{H}_n\backslash \Omega$ has nontrivial stabilizer in $G$, then $G_P^{(0)}=G_P^{(1)} \cong \Z_3$ and $G_P^{(i)}=\{1\}$ for $i \geq 2$, where $G_P^{(i)}$ denotes the $i$-th ramification group of $P$ in $G$.
\end{prop}
\begin{proof}
If $P$ is a point as in the statement, then by Proposition \ref{pf3} we may assume that $P=(1:1:1)$. Recall that the stabilizer of $P$ is $G_P=\langle \mu \rangle \cong \Z_3$. Set $F=XY^n+YZ^n+X^nZ$. Then $\kk(x,y)$ is the function field of $\mathcal{H}_n$, where $x=X/Z \mod F$ and $y=Y/Z \mod F$. Since the tangent line to $\mathcal{H}_n$ at $P$ is given by $X+Y+Z=0$, we have that $t=x-1$ is a local parameter at $P$. Hence $\mu(x)=y/x$ and 
$$
\mu(t)-t= \mu(x)-x=\frac{y-x^2}{x}.
$$
Consider the local expansions of $x=1+t$ and $y=1-t+a_2t^2+a_3t^3+\cdots$ at $P$. Then
\begin{equation}\label{val}
v_P\left(\frac{y-x^2}{x}\right)=v_P((a_2-1)t^2+a_3t^3+\cdots ).
\end{equation}
If $n \equiv 0 \mod 3$, then $P$ is an inflection point (by Theorem \ref{TeoWei}), which means that $v_P(y+x+1)=v_P(y+t-1) >2$. Hence $a_2=0$, which gives $v_P(\mu(t)-t)=2$ by \eqref{val}.
Assume now $n \equiv 1 \mod 3$. In this case, $a_2 \neq 0$. By the local expansion of $y$ at $P$, we have that $a_2=D_t^{(2)}(y)(P)$. Since
$$
D_t^{(2)}(y)=\frac{ny^{n-1}(y^n+n(t+1)^{n-1})}{(n(t+1)y^{n-1}+1)^2},
$$
we obtain $a_2=\frac{n}{n+1} \neq 1$, and so $v_P(\mu(t)-t)=2$. 
\end{proof}

Now we are in a position to present the list of possible genera of the quotients of $\mathcal{H}_n$.

\begin{thm}
The list of all possible values of $g(\mathcal{H}_n/H)$, where $H \leq G$, is given below.
\begin{itemize}
\item[(a)] Case $n \not\equiv 2 \mod 3$:
\begin{enumerate}
\item   $g(\mathcal{H}_n/T_{n^2-n+1})=\frac{n^2-n}{6}$.
   \item  $g(\mathcal{H}_n/S_d)=\frac{n^2-n+1-d}{2d}$, where $d | n^2-n+1$.
	\item $g(\mathcal{H}_n/T_{n^2-n+1}\cdot S_d)=\frac{n^2-n+1-d}{6d}$, where $d | n^2-n+1$.
	\end{enumerate}
\item [(b)] Case $n \equiv 2 \mod 3$:
\begin{enumerate}
   \item  $g(\mathcal{H}_n/T_{0})=\frac{n^2-n-2}{6}$.
   \item  $g(\mathcal{H}_n/T_{n^2-n+1})=\frac{n^2-n+4}{6}$.
	\item   $g(\mathcal{H}_n/T_{\frac{n^2-n+1}{3}})=g(\mathcal{H}_n/T_{2\frac{n^2-n+1}{3}})=\frac{n^2-n-2}{6}$.
	\item   $g(\mathcal{H}_n/S_d)=\frac{n^2-n+1-d}{2d}$, where $d| n^2-n+1$.
	\item   $g(\mathcal{H}_n/T_{n^2-n+1}\cdot S_d)=\frac{n^2-n+1+3d}{6d}$, where $d|n^2-n+1$ and $3 \nmid d$.
	\item   $g(\mathcal{H}_n/T_{\frac{n^2-n+1}{3}}\cdot S_d)=g(\mathcal{H}_n/T_{2\frac{n^2-n+1}{3}}\cdot S_d)=\frac{n^2-n+1-3d}{6d}$, where $d|n^2-n+1$ and $3 \nmid d$.
	\item $g(\mathcal{H}_n/T_{n^2-n+1}\cdot S_d)=\frac{n^2-n+1-d}{6d}$, where $d|n^2-n+1$ and $3 |d$.
		\end{enumerate}
\end{itemize} 
\end{thm}
\begin{proof}
The fact that the list above is exhaustive follows from Proposition \ref{subgroups}.
The genera follow from Lemmas \ref{pfc} and \ref{ost}, Propositions \ref{pf3}, \ref{pf3-2}, \ref{rg} and the Riemann-Hurwitz genus formula \eqref{rh}. For the sake of illustration, we compute $g(\mathcal{H}_n/T_{n^2-n+1})$ for $p=3$ with $n \not\equiv 2 \mod 3$, and $g(\mathcal{H}_n/T_{n^2-n+1}\cdot S_d)$ where $d\mid (n^2-n+1)$, $n \equiv 2 \mod 3$  and $3|d$ (note that in this case, $p \neq 3$). The other cases are analogous. 

Suppose $p=3$ and $n \not\equiv 2 \mod 3$. By Proposition \ref{pf3}, there is only one point with nontrivial stabilizer in $T_{n^2-n+1}$, namely $P=(1:1:1)$. Moreover, Proposition \ref{rg} gives $|(T_{n^2-n+1})_P^{(0)}|=|(T_{n^2-n+1})_P^{(1)}|=3$, and $|(T_{n^2-n+1})_P^{(i)}|=1$ for $i \geq 2$. Hence, the Riemann-Hurwitz genus formula \eqref{rh} gives
$$
2g(\mathcal{H}_n)-2=3(2g(\mathcal{H}_n/T_{n^2-n+1})-2)+\sum_{i \geq 0}\big(|(T_{n^2-n+1})_P^{(i)}|-1 \big),
$$
and so $g(\mathcal{H}_n/T_{n^2-n+1})=\frac{n^2-n}{6}$.

Now assume $n \equiv 2 \mod 3$, and let $d$ be a divisor of $n^2-n+1$ such that $3|d$. The set $\Omega$ is an orbit of size $3$ of $H=T_{n^2-n+1}\cdot S_d$. Via Proposition \ref{pf3-2}, the orbits of both points $Q_1=(\alpha:1:1)$ and $Q_2=(\alpha^2:1:1)$ by $H$ have size $d$, and the remaining orbits by $H$ are long. Since $p \nmid 3d$, the Riemann-Hurwitz genus formula \eqref{rh} provides
$$
2g(\mathcal{H}_n)-2=3d(2g(\mathcal{H}_n/T_{n^2-n+1}\cdot S_d)-2)+3d-3+2(3d-d).
$$
Hence $g(\mathcal{H}_n/T_{n^2-n+1}\cdot S_d)=\frac{n^2-n+1-d}{6d}$.
\end{proof}

\section{An application: nonisomorphic families of plane nonsingular maximal curves }

In this section, we present a further application of our results. More in detail, we present a new family of plane nonsingular maximal curves $\mathcal{C}_n$. Here, \emph{new} means  that $\mathcal{C}_n$ is not isomorphic neither to the Fermat curve $\cF_n$ nor to the Hurwitz curve $\cH_n$.   In particular, we answer a question raised by Fernando Torres during the "Workshop on Algebraic curves and Function Fields over a Finite Field" held in Perugia in February 2015. 

\begin{defi}
Let $\rm{char}(\kk) \neq 2$.  Let $\kk[x]$ be the polynomial ring of univariate polynomials with coefficients in $\kk$. We define the $n$-th Lucas-type polynomial $L_n(x)$ as 
\begin{itemize}
\item $L_0(x) = 2$;
\item $L_1(x) = x$;
\item $L_n(x) = xL_{n-1}(x)-L_{n-2}(x)$. 
\end{itemize}
\end{defi}

A fundamental property of the Lucas-type polynomial  $L_n(x)$ is given by the following Lemma.
\begin{lem}\label{lem:fund}
For $n \geq 2$, the $n$-th Lucas-type polynomial satisfies
        \begin{equation}\label{FundProp}
        L_n\Bigl(x+\frac{1}{x}\Bigr)=x^n+\frac{1}{x^n}.
        \end{equation}
\end{lem}
\begin{proof}
By induction on $n$. For $n = 2$, a straightforward computation yields 
$$
L_2\Bigl(x+\frac{1}{x}\Bigr) = x^2+\frac{1}{x^2}.
$$
Next, assume $n > 2$ and  that property \eqref{FundProp} holds for any $k < n$. Then, we get
$$
L_n\Bigl(x+\frac{1}{x}\Bigr)= \Bigl(x+\frac{1}{x}\Bigr)L_{n-1}\Bigl(x+\frac{1}{x}\Bigr)-L_{n-2}\Bigl(x+\frac{1}{x}\Bigr).
$$
By the induction hypothesis, the latter equality reads
$$
L_n\Bigl(x+\frac{1}{x}\Bigr) = \Bigl(x+\frac{1}{x}\Bigr)\Bigl(x^{n-1}+\frac{1}{x^{n-1}}\Bigr)- \Bigl(x^{n-2}+\frac{1}{x^{n-2}}\Bigr),
$$
whence our assertion follows.
\end{proof}

\begin{prop}
Let ${\rm char}(\kk) = p >2$, and let $n$ be a divisor of $\frac{p^r+1}{2}$, $r\geq 1$. Then the curve $\mathcal{C}_n$ given by the affine equation
$$
\mathcal{C}_n : y^n=L_n(x)
$$
is a smooth plane $\mathbb{F}_{p^{2r}}$ maximal curve.
\end{prop}
\begin{proof}
Note that, if $p \nmid n$, Lemma \ref{lem:fund} ensures that $L_n(x)$ is separable. Hence, $\mathcal{C}_n$ is irreducible and nonsingular. By Lemma \ref{lem:fund}, we have that $\mathcal{C}_n$ is a subcover of the Generalized Fermat curve $\mathcal{G}_n$ of affine equation $\mathcal{G}_n: y^n=x^{2n}+1$. More in detail, let $\psi: \mathcal{G}_n \rightarrow \mathcal{C}_n$ being given by 
$$
\psi(x,y) = \Bigl( x+\frac{1}{x}, \frac{y}{x}\Bigr).  
$$
Then by Lemma \ref{lem:fund}, it follows that 
$$
L_n\Bigl(x+\frac{1}{x}\Bigr)= \frac{x^{2n}+1}{x^n} = \frac{y^n}{x^n},
$$
whence $\mathcal{C}_n$ is covered by $\mathcal{G}_n$. 

 Finally, by \cite[Theorem 5]{TT}, $\mathcal{G}_n$ is $\mathbb{F}_{p^{2r}}$-maximal if and only if $n \mid p^{r}+1$, whence our assertion follows. 
        \end{proof}

\begin{prop}\label{flexes}
 Let $\mathbb{K}$ be an algebraically closed field of characteristic $p > 2$, and let
$$\mathcal{C}_n:\, Y^n-L_n(X,Z)=0$$
  be a  plane curve defined over $\kk$ where  $p\nmid 2n$, $n \mid {\frac{p^r+1}{2}}$, and  $L_n(X,1)$ is the $n$-th Lucas-type polynomial. Then      $\mathcal{C}_n$  has either $n$  or $3n$ total inflection points. If $n>3$, then the  latter case occurs  if and only if $\mathcal{C}_n$ is  projectively equivalent  to  the Fermat curve  $$Y^{\frac{p^r+1}{2}}+ X^{\frac{p^r+1}{2}}+Z^{\frac{p^r+1}{2}}=0,$$
where $r\ge1 $.
\end{prop}

\begin{proof}
Let $\mathcal{I} = \{P_i=(x_i:0:1) \mid L_n(x_i,1)=0, i=1,\cdots,n\}$. It is immediately seen that the points in $\mathcal{I}$ 
%It is clear that the points $P_i=(x_i:0:1)$,  where $L_n(x_i,1)=0$,  for $i=1,\cdots,n $
 are  total inflection points of $\mathcal{C}_n$ whose tangent lines are $X=x_iZ$, and  that no other inflection point of $\mathcal{C}_n$ has tangent line of this latter type. Let $P\in \mathcal{C}_n \backslash \mathcal{I}$ be a total inflection point of $\mathcal{C}_n$, and let $\ell$ be the corresponding tangent line. From  \eqref{FundProp}, it is easy  to  check that        $\ell$  cannot  be of type  $bY+cZ=0$,  and thus  $\ell$  has an affine  equation  of type   $Y=aX+bZ$, where  $a\neq 0$. A computation shows that the $2d$ points in $\mathcal{C}_n \cap \{XZ = 0\}$ are not inflection points. Hence, let $P=(x_0:y_0:1)$, with $x_0y_0 \neq 0$ be a further total inflection point for $\mathcal{C}_n$. Then  we have the polynomial identity  
       
       $$(ax+b)^n-L_n(x)=(a^n-1)(x-x_0)^n,$$
       
       where $a^n\neq 1$. Also,  the substitution $x \rightarrow x+\frac{1}{x}$ yields 
       
       \begin{equation}\label{poly-id}
       a^n(x^2+\frac{b}{a}x+1)^n+(1-a^n)(x^2-x_0x+1)^n=x^{2n}+1,
       \end{equation}
       as $ L_n\Bigl(x+\frac{1}{x}\Bigr)= \frac{x^{2n}+1}{x^n}$. 
        
In the expansion of $(x^2+tx+1)^{n}$, note that   the coefficients of $x^{2n-1}$, $x^{2n-2}$, $x^{2n-3}$ and $x^{2n-4}$
are respectively $nt$,  $n +\binom{n}{2}t^2,$  $2\binom{n}{2}t+\binom{n}{3}t^3$, and  $\binom{n}{2}+ 3\binom{n}{3}t^2 + \binom{n}{4}t^4.$  After  comparing coefficients in  \eqref{poly-id}, as $n>3$ and   $p>0$, we conclude that $a^n=\frac{1}{2}$     and  $\frac{b}{a}=x_0=2$. Thus \eqref{poly-id}
becomes 
$$ (x+1)^{2n}+(x-1)^{2n}=2 x^{2n}+2.$$
Differentiating both sides of  the equation above, and  writing  $2n-1=kp^r$, with $p\nmid k$, we obtain the  identity
 \begin{equation}\label{fermat-id}
   (x^{p^r}+1)^k+(x^{p^r}-1)^{k}=2 (x^{p^r})^k.
  \end{equation}
Clearly   \eqref{fermat-id} does not  hold for $k=2$, and since the smooth  curve   $X^k+Y^k=2Z^k$ is not rational for $k\geq 3$, it follows that $k=1$. Therefore,  $n=\frac{p^r+1}{2}$. The last assertion follows from the $\F_{p^{2r}}$-maximality of   $\mathcal{C}_n$, and from  the classification of  maximal curves of genus $g = g(\mathcal{C}_n)$, see \cite[Theorem 1.1]{CHKT}.% for $p^r \geq 11$
\end{proof}

\begin{rem}
For $n = 3$, it is easily seen that the curves $\cF_3$ and $\mathcal{C}_3$ are isomorphic via a fractional transformation.
\end{rem}

\begin{cor} Let $q$ be a power of a  prime  $p>2$. For any divisor $n$ of $q+1$,  with  $3<n<\frac{q+1}{2}$, there exist a  smooth
plane $\F_{q^2}$-maximal curve $\mathcal{C}$   such that neither a Fermat $\mathcal{F}_n$ nor  Hurwitz curve $\mathcal{H}_{n-1}$ of degree $n$  is isomorphic  to  
$\mathcal{C}$.
\end{cor}
\begin{proof}
Clearly, we have $\mathcal{C} = \mathcal{C}_n$. On the one hand, Proposition \ref{flexes} ensures that $\mathcal{C}_n$ is not isomorphic to $\mathcal{F}_n$. On the other hand, it is enough to observe that $\mathcal{H}_n$ has no total inflection points, and again the result follows by Proposition \ref{flexes}. 
\end{proof}
\section*{Acknowledgements}

The first author  was supported by FAPESP-Brazil, grant 2016/24713-4.
The second  author  was supported by FAPESP-Brazil, grant 2017/04681-3.
The third author  was supported by FAPESP-Brazil, grant 2017/18776-6.

\printindex

\end{document}